\documentclass[a4paper,12pt]{amsart}
\usepackage[centertags]{amsmath}
\usepackage{amssymb}
\usepackage{amsthm, float}
\usepackage[english]{babel}
\usepackage[latin1]{inputenc}
\usepackage{fancyhdr}
\usepackage[top=2.0cm,bottom=2.50cm,left=1.5cm,right=1.5cm]{geometry}
\usepackage[T1]{fontenc}
\usepackage{xcolor}
\usepackage{fancybox}
\usepackage{graphicx}
\usepackage{pifont}
\usepackage{tablists}
\usepackage{multicol}
\usepackage{mathrsfs}
\usepackage{lastpage, hyperref}
\usepackage{etoolbox}
\makeatletter
\@namedef{subjclassname@2020}{%
	\textup{2020} Mathematics Subject Classification}
\makeatother

\theoremstyle{plain}

\newtheorem{conjecture}[subsection]{Conjecture}
\theoremstyle{definition} 
\newtheorem{theorem}[subsection]{Theorem}
\newtheorem{definition }[subsection]{Definition}
\newtheorem{lemma}[subsection]{Lemma}

\newtheorem{example}[subsection]{Example}
\newtheorem{remark}[subsection]{Remark}
\newtheorem{notation}[subsection]{Notation}

\newcommand{\R}{{\mathbb R}}  
\newcommand{\N}{{\mathbb N}}  
\newcommand{\Def}[1]{\textbf{\boldmath{#1}}}
\DeclareMathOperator*{\colim}{colim}
\DeclareMathOperator*{\bast}{\divideontimes}
\DeclareMathOperator*{\bbast}{\overline{\divideontimes}}

\begin{document}
	\title{Homeomorphic Model for the Polyhedral Smash Product of Disks and Spheres}															
	\author{Arnaud Ngopnang Ngompe}
	\address{University of Regina, 3737 Wascana Pkwy, Regina, SK S4S 0A2, Canada}
	\email{\url{ann037@uregina.ca}}
	
	\begin{abstract} In this paper we present unpublished work by David Stone on polyhedral smash products. He proved that the polyhedral smash product of the CW-pair $(D^2, S^1)$ over a simplicial complex $K$ is homeomorphic to an iterated suspension of the geometric realization of $K$. Here we generalize his technique to the CW-pair $(D^{k+1}, S^{k})$, for an arbitrary $k$.  We generalize the result further to a set of disks and spheres of different dimensions.
	\end{abstract}
	
	\keywords{polyhedral product, smash product, simplicial complex, geometric realization, join}
		
	\subjclass[2020]{Primary 55U10; Secondary 57Q05}

	\maketitle
	
	\tableofcontents
	
	\section{Introduction}
	
	In all the following, $m\in \N$ is any natural number and $[m]=\{1,\cdots,m\}$. Also, we set $K$ to be an abstract simplicial complex whose vertex set is contained in $[m]$, that is $K$ is a family of subsets $\sigma\subseteq [m]$, called simplices, such that whenever $\sigma\in K$ and $\tau\subseteq \sigma$, then $\tau \in K$.
	
	\begin{definition }[Polyhedral smash product]\cite[\textbf{Construction~8.3.1}]{bp2} Let $(\underline{X},\underline{A})=\{(X_i,A_i)\}_{i\in [m]}$ be a family of pointed CW-pairs, that is, the $X_i$ are CW-complexes and $A_i\hookrightarrow X_i$ are subcomplexes, for all $i\in [m]$. The \Def{polyhedral smash product} of $(\underline{X},\underline{A})$ over $K$, denoted $\widehat{Z}(K; (\underline{X},\underline{A}))\subseteq \bigwedge_{i=1}^m X_i$, is the space given by
		\[\widehat{Z}(K; (\underline{X},\underline{A}))=\bigcup_{\sigma\in K}\widehat{D}(\sigma),\]
		
	\begin{equation}\label{dhat}
			\text{where }\widehat{D}(\sigma)=\bigwedge_{i=1}^m Y_i \quad\text{with}\ Y_i=\begin{cases}X_i\ \text{if}\ i\in \sigma,\\
			A_i\ \text{otherwise}
		\end{cases},\quad \forall \sigma\in K.
	\end{equation}
Using categorical language, consider $\textsc{Cat}(K)$ to be the face category of $K$, that is, objects are simplices and morphisms are inclusions. Define the $\textsc{Cat}(K)$-diagram given by
\begin{align}\label{dcat}
	\widehat{D}: \textsc{Cat}(K) & \to \text{Top} \nonumber\\
	\sigma  & \mapsto \widehat{D}(\sigma),
\end{align}
where $\widehat{D}(\sigma)$ is given by $(\ref{dhat})$ and the functor $\widehat{D}$ maps the morphism $\rho\subseteq \sigma$ to the inclusion $\widehat{D}(\rho)\subseteq \widehat{D}(\sigma)$. Then
\begin{equation}\label{psp}
	\widehat{Z}(K; (\underline{X},\underline{A})) = \colim_{\sigma \in K} \widehat{D}(\sigma).
\end{equation}
	\end{definition }
	
	 Below we recall some well-known operations on spaces.
	\begin{definition }\cite[\textbf{\S 0}]{AH}  \text{\ }
	For $n\in \N$, let $(X,x_0)$ and $(Y,y_0)$ be two pointed topological spaces.
		\begin{itemize}
			\item The \Def{join} $X\ast Y$ of $X$ and $Y$ is the quotient space defined by $X\ast Y=X\times Y\times I/\sim,$ where $I=[0,1]$ and $\sim$ is the equivalence relation generated by
			\begin{align*}
				(x,y,0) & \sim (x,y',0),\ \forall x\in X\ \text{and}\ \forall y,y'\in Y,\\
				(x,y,1) & \sim (x',y,1),\ \forall x,x'\in X\ \text{and}\ \forall y\in Y.
			\end{align*}
			\item The \Def{wedge sum} $X\vee Y$ of $X$ and $Y$ is the quotient space defined by $X\vee Y= X\amalg Y/(x_0\sim y_0)$.
			\item The \Def{smash product} $X\wedge Y$ of $X$ and $Y$ is the quotient space defined by $X\wedge Y=X\times Y/X\vee Y$.
			\item The \Def{(unreduced) suspension} $\Sigma X$ of $X$ is the space defined by $\Sigma X =S^0 \ast X$, where $S^0$ denotes the $0$-sphere.
			\item The \Def{(unreduced) cone} $CX$ of $X$ is the space defined by $CX=c\ast X$, where $c$ is a single point.
		\end{itemize}
	\end{definition }
	
	David stone made the following conjecture.
	\begin{conjecture}\label{conj}
		If $F$ is a compact subspace of $\R^n$, then there is a homeomorphism \begin{equation*}
			\widehat{Z}(K; (c\ast \Sigma F,\Sigma F))\cong \Sigma \left( \bast^m F\right) \ast|K|, 
		\end{equation*}
		where $\bast^m F$ is defined as the $m$-fold join of $m$ copies of $F$.
	\end{conjecture}
	
	As it is mentioned in \cite[\textbf{Remark~2.20}]{bbcg1}, David Stone used a kind of geometrical argument to prove a particular case of his conjecture by taking $F=S^0$. Hence he proved the following.
	\begin{theorem} \cite{david}\label{stone} There is a homeomorphism
	\begin{equation*}
			\widehat{Z}(K; (D^{2},S^1))\cong \Sigma^{m+1} |K|.
	\end{equation*}
	\end{theorem}
	
	In this paper we apply the same technique to a more general case. For $k\in \N$, we consider $F=S^{k-1}$, which is compact (as a closed and bounded subspace of $\R^k$), and we have $\Sigma F\cong S^k$, $c\ast\Sigma F\cong D^{k+1}$ and $\bast^m F\cong S^{km-1}$ (since $S^i\ast S^j\cong S^{i+j+1}$). Hence
	\begin{align*}
		\left( \bast^m F\right) \ast|K| & \cong  S^{km-1}\ast|K|\\
		& \cong \left(\bast^{km}S^0\right) \ast |K|\\
		& \cong \Sigma^{km} |K|.
	\end{align*}
	
	So we can state a generalization of Stone's result. 
	
	\begin{theorem}\label{main} For any $k\in \N\cup\{0\}$, there is a homeomorphism
		\begin{equation*}
			\widehat{Z}(K; (D^{k+1},S^k))\cong \Sigma^{km+1} |K|.
		\end{equation*} 
	\end{theorem}
	

The goal of this paper is first to generalize David Stone's technique for the proof of \textbf{Theorem \ref{main}} and secondly to provide a further generalization (see \textbf{Theorem \ref{gen}}) of the latter result for a set of disks and spheres of different dimensions. 

\begin{theorem}\label{maingen} For any $m$-tuple $J=(j_1,\cdots,j_m)$ in $(\N\cup\{0\})^m$, there is a homeomorphism
	\begin{equation*}
		\widehat{Z}(K;(\underline{D^{J+1}},\underline{S^J}))\cong \Sigma^{j_1+\cdots+j_m+1} |K|,
	\end{equation*}
where $\left( \underline{D^{J+1}},\underline{S^J}\right) =\left\lbrace \left( D^{j_1+1}, S^{j_1}\right) ,\cdots,\left( D^{j_m+1},S^{j_m}\right) \right\rbrace$.
\end{theorem}

In order to prove \textbf{Theorem \ref{main}} we need to put together some topological and combinatorial tools, hence the rest of the paper is organized as follows. In \textsc{Sections} \ref{two} and \ref{three}, we describe respectively the necessary topological and combinatorial tools. \textsc{Section} \ref{four} is devoted to the proof of \textbf{Theorem \ref{main}}, for $k\geq 1$. The case $k=0$, namely \textbf{Theorem \ref{zero}}, is treated in \textsc{Section} \ref{five} using a more categorical argument. Finally,  in \textsc{Section \ref{six}} we prove the main result, \textbf{Theorem \ref{gen}}, using an inductive argument based on the case $k=0$.
	\vspace{.5cm}
	
		\paragraph*{\textbf{Acknowledgments}} 
	I would like to thank the Fields Institute at the University of Toronto where I started this work during the Thematic Program on Toric Topology and Polyhedral Products, for the inspiring work environment and the valuable financial support. I am thankful to Anthony Bahri for sharing the private correspondence from David Stone with me. Also I would like to express my deep gratitude to both Donald Stanley and Martin Frankland for their scientific and financial contribution to the realization of this work. Finally, I am deeply grateful to the referee for the concrete idea and guideline regarding the methodology that helped me to build the entire \textsc{Section} \ref{six}.
	
	\section{Topological tools}\label{two}
	Among the tools we use in the proof of \textbf{Theorem \ref{main}}, the homeomorphisms $\Psi: C\Delta^{n-1}  \to C^n$ and $\overline{\Psi}:\Sigma\Delta^{n-1}\to \widetilde{D}^n$, described below, are both playing an important role. They were defined by David Stone in \cite{david} and we recycle them here to prove this more general case. Before we introduce them, let us first recall the usual homeomorphism $\widetilde{\Theta}: CX/X\to \Sigma X$.
	
	Given a space $X$, we identify $X$ with the base $c\times X\times \{1\}$ of $CX=c\ast X$. Set $S^0=\{s_1,s_2\}$ to be the $0$-sphere and consider the map \begin{align*}
		\Theta: CX & \to \Sigma X \\
		[c,x,\lambda] & \mapsto \Theta[c,x,\lambda]=\begin{cases}
			(s_1, x, 2\lambda),\ \text{if}\ 0\leq \lambda\leq \dfrac{1}{2}\\
			(s_2,x,2-2\lambda),\ \text{if}\ \dfrac{1}{2}\leq \lambda\leq 1.
		\end{cases}
	\end{align*}
	Then $\Theta$ factors through a map $\widetilde{\Theta}: CX/X\to \Sigma X$; see \textsc{Figure} \ref{theta}.
	
	\begin{figure}[h]
		\begin{center}
			\includegraphics[scale=.8]{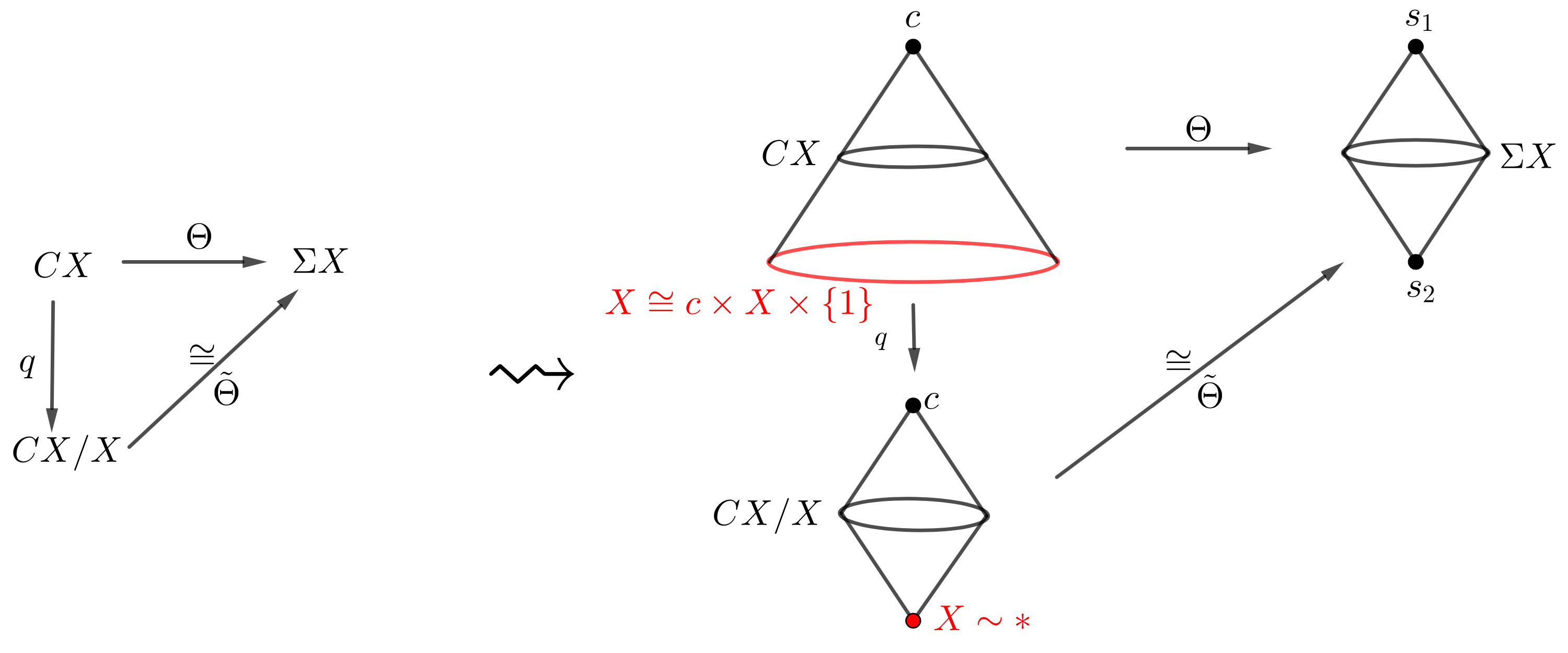}
		\end{center}
		\caption{Factorization of $\Theta$ through $\widetilde{\Theta}$.}
		\label{theta}
	\end{figure}
	
	\begin{lemma}\label{the}
		The map $\widetilde{\Theta}$ is a homeomorphism.
	\end{lemma}
	
	\begin{notation}\label{CD}
		In $\R^n$, 
		\begin{itemize}
			\item let $e_i$ be the $i^\text{th}$ standard basis vector. Let $c$ denote the origin and let $t_1,\cdots,t_n$ denote the coordinates of a point $x\in \R^n$. We identify $x\sim \overrightarrow{cx}$ and so $x=\sum_{i=1}^n t_i e_i$.
			\item Set $C=[0,2]$, with based point $2$ and consider
			
			$C^n=[0,2]^n=\left\lbrace \sum_{i=1}^n t_i e_i\in \R^n: 0\leq t_i\leq 2\right\rbrace $, the $n$-cube of side $2$.
			
			$\partial_+ C^n=\left\lbrace \sum_{i=1}^n t_i e_i\in C^n: \max t_i=2\right\rbrace $, the outer boundary of $C^n$.
			
			$\partial_- C^n=\left\lbrace \sum_{i=1}^n t_i e_i\in C^n: \min t_i=0\right\rbrace $, the inner boundary of $C^n$.
			
			$\partial C^n=\partial_+ C^n \cup \partial_- C^n$, the boundary of $C^n$.
			
			$\widetilde{D}^n=C^n/\partial_+ C^n$, with the quotient map $\omega: C^n\to \widetilde{D}^n$.
		\end{itemize}
	\end{notation}
	\begin{lemma}\label{nsdisk}
		The quotient space $\widetilde{D}^n$ is a topological disk.
	\end{lemma}
	
	\begin{proof}
		
		Considering the CW-pair $(C^n, \partial_+C^n)$, we have
		\begin{align*}
			(C^n, \partial_+C^n) & \cong (c \ast \partial_+C^n, \partial_+C^n) \\
			& \cong (c \ast D^{n-1}, D^{n-1}) ,
		\end{align*} 
		where the latter pair is the inclusion of the base of the cone $c\times D^{n-1} \times \{1\}\cong D^{n-1}\subseteq (c\ast D^{n-1})$. Hence collapsing the respective subspaces yields a homeomorphism
		\begin{align*}
			\widetilde{D}^n= C^n / \partial_+ C^n & \cong (c\ast D^{n-1}) / D^{n-1}\\
			& \cong \Sigma D^{n-1} \\
			&  \cong D^n.
		\end{align*}
Therefore $\widetilde{D}^n$ is a topological disk.
\end{proof}

	\begin{notation}\label{ch} Let us consider the following setup.
		
		\begin{itemize}
			\item For any set $X=\{x_1,\cdots, x_p\}\subseteq \R^n$, let $\text{cx}(X)$ denote the convex hull of $X$, that is
			\[\text{cx}(X)=\left\lbrace \sum_{i=1}^p t_i x_i\in \R^n: t_i\geq 0, \sum_{i=1}^p t_i =1\right\rbrace .\]
			
			\item Set $\Delta^{n-1}=\text{cx}\{e_1,\cdots,e_n\}$ to be the standard $(n-1)$-simplex.
			
			\item For any $J\subseteq [n]$, set $\Delta(J)=\text{cx}\left( \{e_i: i\in J\}\right)\cong \Delta^{|J|-1}$, where $|J|$ denotes the cardinality of $J$.
		\end{itemize}
	\end{notation}

	\begin{remark}
	The abstract cone $C\Delta^{n-1}$ can be realized as a subspace of $\R^n$, a subspace which is homeomorphic to the $n$-cube $C^n$ by reparametrization as we can observe in \textsc{Figure} \ref{Psi}. This motivates the existence of a bijection $\Psi:C\Delta^{n-1} \to C^n$, defined by equation $(\ref{Psidef})$.
	\begin{figure}[h]
	\begin{center}
		\includegraphics[scale=.8]{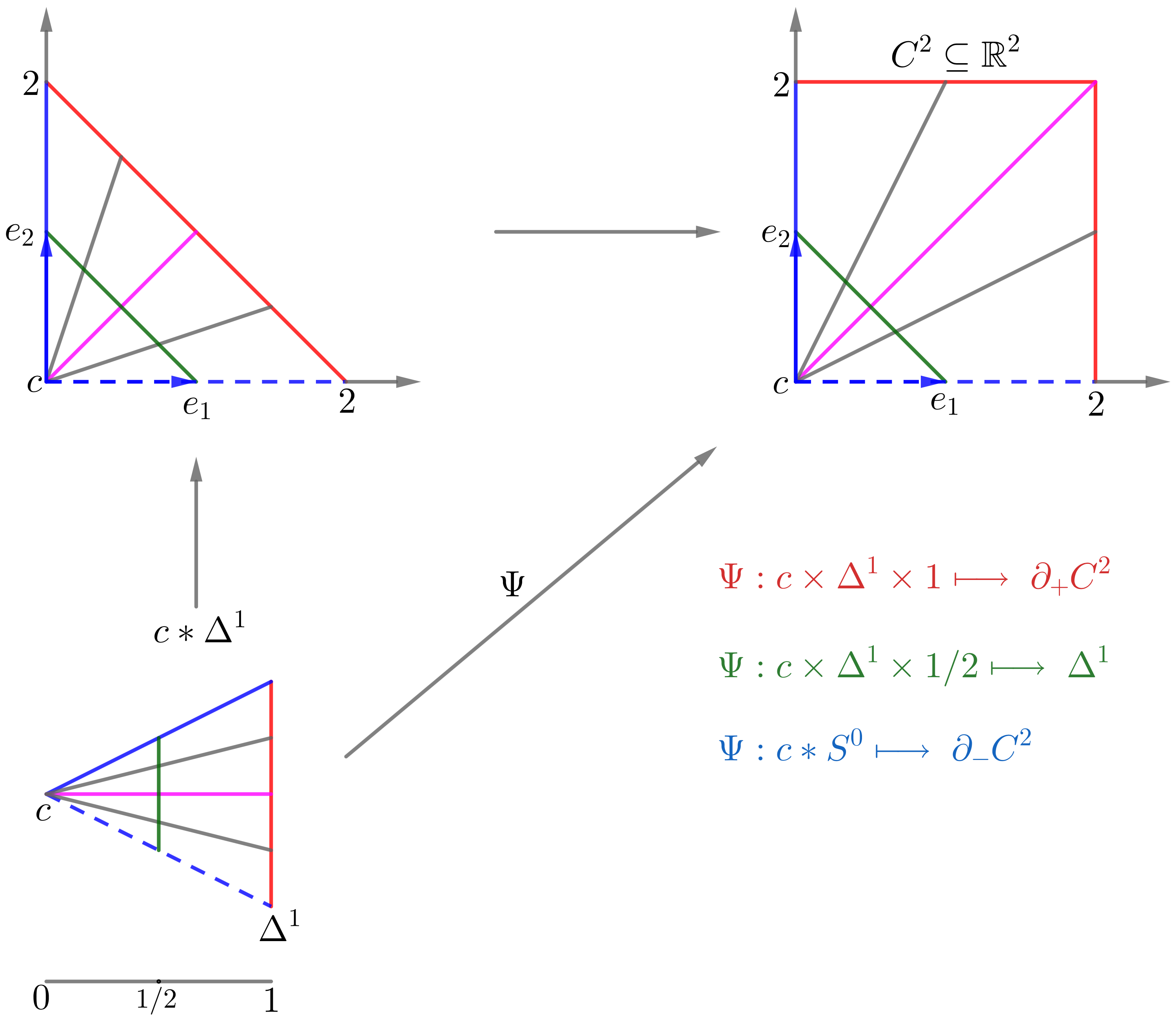}
	\end{center}
	\caption{Illustration of the map $\Psi: C\Delta^{1} \to C^2$, that is for $n=2$.}
	\label{Psi}
	\end{figure}
	
	\end{remark}
	For $x=\sum_{i=1}^n t_i e_i\in \Delta^{n-1}$, set $\overline{t}=\max \{t_i\}$, so $\overline{t}>0$. Define the map
	\begin{align}\label{Psidef}
		\Psi: C\Delta^{n-1} & \to C^n \\
		[c,x,\lambda] & \mapsto \Psi[c,x,\lambda]=\begin{cases}
			2\lambda x,\ \text{if}\ 0\leq \lambda\leq \dfrac{1}{2} \nonumber\\
			\left( (2-2\lambda)+(2\lambda-1)\dfrac{2}{\overline{t}}\right) x,\ \text{if}\ \dfrac{1}{2}\leq \lambda\leq 1,
		\end{cases}
	\end{align}
where $C^n=[0,2]^n$ is the $n$-cube of side $2$ set in \textbf{Notation \ref{CD}}.

\begin{remark}
As mentioned in \textbf{Notation \ref{CD}}, the basepoint of $C=[0,2]$ is $2$. The above defined map $\Psi$ does not send the cone point to the basepoint $(2,\cdots,2)$ of the $n$-cube $C^n$, as one might expect, but to the origin $c$ of $\R^n$ for convenience.	
\end{remark}

	By \textbf{Lemma \ref{the}}, we have $C\Delta^{n-1}/\Delta^{n-1}\cong \Sigma \Delta^{n-1}\cong \Delta^n$. Also $\Psi(c\times \Delta^{n-1}\times \{1\})=\partial_+C^n$ and hence, $\Psi$ factors through the map 
	\[\overline{\Psi}:\Sigma\Delta^{n-1}\to \widetilde{D}^n,\]
	where $\widetilde{D}^n=C^n/\partial_+ C^n$ is the topological disk introduced in \textbf{Notation \ref{CD}}.
	\begin{center}
		\includegraphics[scale=.75]{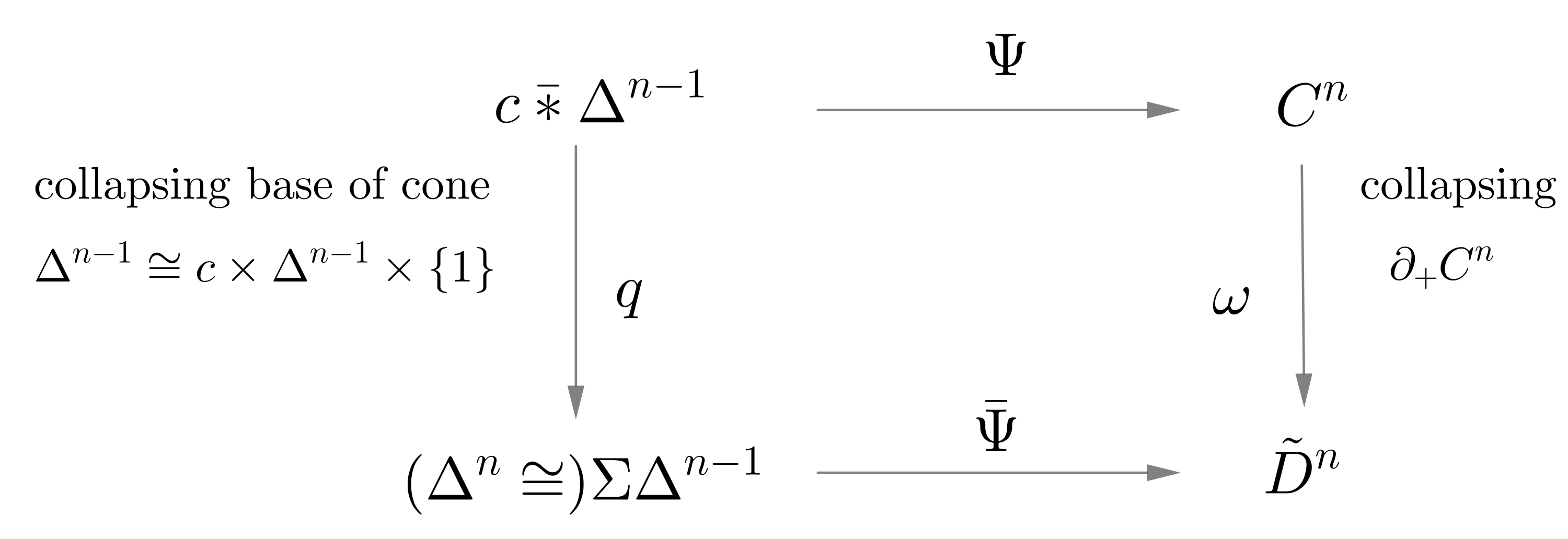}
	\end{center}
	
	\begin{lemma}\label{psi}
		The maps $\Psi$ and $\overline{\Psi}$ are both homeomorphisms.
	\end{lemma}
	
	\begin{proof} As a continuous bijection from the compact space $C\Delta^{n-1}$ to the Hausdorff space $C^n$, $\Psi$ is a homeomorphism. Hence, $\Psi$ gives us the homeomorphism of the pairs $(C\Delta^{n-1},\Delta^{n-1})\cong (C^n,\partial_+C^n)$, so that the induced map $\overline{\Psi}:\Sigma\Delta^{n-1}\to \widetilde{D}^n$ is a homeomorphism.
	\end{proof}
	
	\begin{remark}\label{om}
		If we consider $\left( [0,2], 2\right) $ to be a pointed space, then collapsing $\partial_+C^{2(k+1)}$ in $C^{2(k+1)}\cong C^{k+1}\times C^{k+1}$, we get $\widetilde{D}^{k+1}\bigwedge\widetilde{D}^{k+1}\cong D^{k+1}\bigwedge D^{k+1}\cong D^{2(k+1)}$. This can be generalized to the case of $C^{p(k+1)}\cong \underbrace{C^{k+1}\times\cdots\times C^{k+1}}_{p \ \text{times}}$ and so collapsing $\partial_+C^{p(k+1)}$ corresponds to $\bigwedge^p \widetilde{D}^{k+1}\cong \bigwedge^p D^{k+1}\cong D^{p(k+1)}$. Hence 
		
		\begin{align*}
			&\omega\left( \prod^p C^{k+1}\right) \cong \bigwedge^p \widetilde{D}^{k+1}\cong \bigwedge^p D^{k+1}\ \text{and}\\ 
			&\omega\left( \prod^p \partial_- C^{k+1}\right) \cong \bigwedge^p \partial \widetilde{D}^k \cong \bigwedge^p S^k,\ \text{where}\ \partial \widetilde{D}^k\ \text{denotes the boundary of}\ \widetilde{D}^{k+1}.
		\end{align*}
		
		\begin{center}
			\includegraphics[scale=.7]{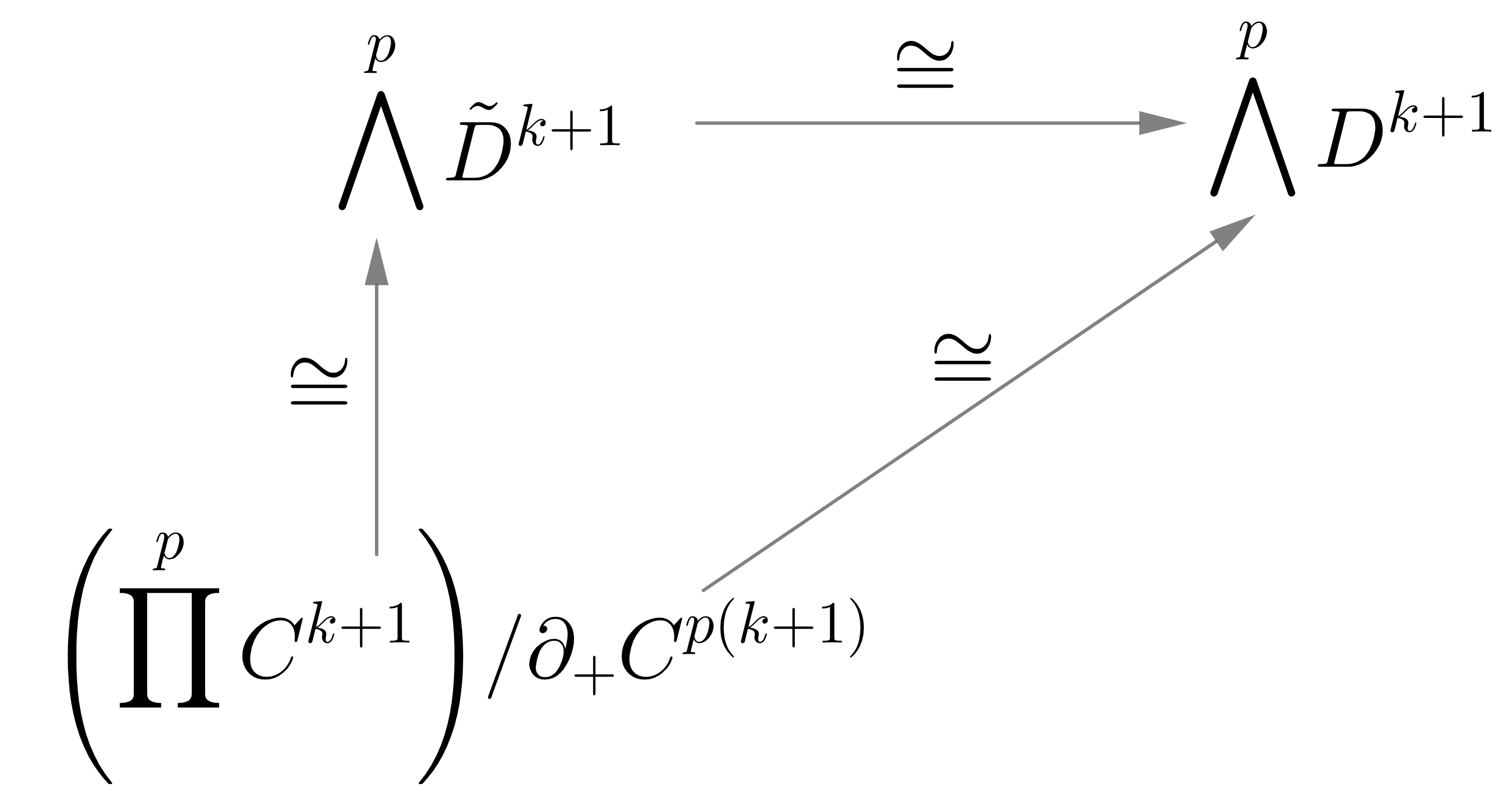}
		\end{center}
	\end{remark}
	
	\begin{lemma}\label{bbcg} For any compact and Hausdorff spaces $X$ and $Y$, there is a homeomorphism 
		\begin{equation*}
			\varphi:C(X\ast Y)\to CX \times CY.
		\end{equation*}
	\end{lemma}
	\begin{proof} We follow the proof of \cite[\textbf{Lemma~8.1}]{bbcg2}. We can represent a point in $C(X\ast Y)$ by $[c,[x,y,t],\lambda]$. We define the homeomorphism $\varphi$ by
	\begin{equation*}
	\varphi([c,[x,y,t],\lambda])=([c,x,2\lambda\cdot \min\{t,1/2\}],[c,y,2\lambda\cdot \min\{1-t,1/2\}])\in CX\times CY,
	\end{equation*}
where the cone point is at $\lambda=0$. At $\lambda=1$, $\varphi$ reduces to the usual homeomorphism
\begin{equation*}
	X\ast Y\cong (CX\times Y)\cup (X\times CY).
\end{equation*}
The map $\varphi$ is a homeomorphism as a continuous bijection from the compact space  $C(X\ast Y)$ to the Hausdorff space $CX \times CY$.
	\end{proof}

	\section{Combinatorial tools}\label{three}
	
	One of the main goals of this section is to embed the simplicial complex $K$ in a bigger simplex with vertex set $[(k+1)m]$. We start by introducing a linearized version of the join of spaces.
	
\begin{definition }[Geometrically joinable]\text{\ }

	\begin{itemize}
	\item Two compact subspaces $X$ and $Y$ of $\R^n$ are said to be \Def{geometrically joinable} if whenever $x,x'\in X$, $y,y'\in Y$ and $\lambda,\lambda'\in I$ are such that $\lambda x+(1-\lambda)y=\lambda' x'+(1-\lambda')y'$,
	then we have one of the three following possibilities
	\begin{itemize}
	\item $\lambda=\lambda'=0$, and so $y=y'$;
	\item $\lambda=\lambda'=1$, and so $x=x'$;
	\item $0\neq \lambda=\lambda'\neq 1$, $x=x'$ and $y=y'$.
	\end{itemize}
	\item More generally, $p$ compact subspaces $X_1,\cdots,X_p\subseteq \R^n$ are \Def{geometrically joinable} if whenever we have an equality between two convex combinations of points of $X_1,\cdots,X_p$, that is, whenever
	\begin{equation*}
	\sum_{i=1}^{p}\lambda_i x_i=\sum_{i=1}^{p}\lambda_i' x_i',
	\end{equation*}
	for some $x_i,x_i'\in X_i$ and $\lambda_i,\lambda_i'\in I$ with $\sum_{i=1}^{p}\lambda_i=1=\sum_{i=1}^{p}\lambda_i'$, then for all $i=1,\cdots,p$ such that $\lambda_i\neq 0$ or $\lambda_i'\neq 0$, we have $\lambda_i=\lambda_i'$ and $x_i=x_i'$.
	\item If $p$ compact subspaces $X_1,\cdots,X_p\subseteq \R^n$ are geometrically joinable, then we define their \Def{geometric join} $\bbast_{i=1}^p X_i$ to be the set of all convex combinations of elements of $X_i$, that is,
	\begin{equation*}
	\bbast_{i=1}^p X_i = \left\lbrace \sum_{i=1}^{p}\lambda_i x_i\in \R^n: x_i\in X_i\ \text{and}\ \lambda_i\in I\ \text{such that}\ \sum_{i=1}^{p}\lambda_i=1\right\rbrace.	
	\end{equation*}
	\end{itemize}
\end{definition }
The notion of geometrically joinable introduced here is also called \textquotedblleft in general position\textquotedblright. In the following, when we use the notation $X\overline{\ast} Y$, it is to be understood that $X$ and $Y$ are indeed geometrically joinable.
	\begin{example} Single points $X_1 = \{x_1\},\cdots,X_p = \{x_p\}$ in $\R^n$ are geometrically joinable if and only if they are affinely independent. In that case, their geometric join $\bbast_{i=1}^p X_i$ is their convex hull, which is a $(p-1)$-simplex, that is,
		\begin{equation*}
		\bbast_{i=1}^p X_i=\text{cx}\{x_1,\cdots,x_p\}\cong \Delta^{p-1}.
		\end{equation*}
	\end{example}
	
	\begin{lemma}\text{\ }
	
	\begin{enumerate}
	\item Let $X$ and $Y$ be geometrical joinable subspaces of $\R^n$. The map 
	\begin{align*}
	\Phi: X\ast Y & \to X\overline{\ast} Y \\
	[x,y,\lambda] & \mapsto \lambda x+(1-\lambda)y
	\end{align*}
	is a homeomorphism.
	\item More generally, for geometrically joinable subspaces $X_1,\cdots, X_p$ of $\R^n$, the map
	\begin{align*}
	\Phi_p: \bast_{i=1}^p X_i & \to \bbast_{i=1}^p X_i \\
	[x_i,\lambda_i]_{i=1}^p & \mapsto \sum_{i=1}^{p}\lambda_i x_i
	\end{align*}
	is a homeomorphism.
	\end{enumerate}
	\end{lemma}
	
	\begin{remark}\label{gjss}
		Observe that if $U$ and $V$ are respective subspaces of geometrically joinable spaces $X$ and $Y$, then $U$ is geometrically joinable to $V$.
	\end{remark}
	
	\begin{lemma}\label{gjlem}
		Let  $X, Y_1$ and $Y_2$ be three compact subspaces of $\R^n$. If $X$ is geometrically joinable to each $Y_i$ and 
		\begin{equation}\label{int2}
			X\overline{\ast} Y_1\cap X\overline{\ast} Y_2 = X\overline{\ast} \left( Y_1\cap Y_2 \right),
		\end{equation}
		then $X$ is geometrically joinable to $Y_1\cup Y_2$.
		
	\end{lemma}
	
	\begin{proof}  Let $x, x'\in X$, $w,w'\in Y_1\cup Y_2$ and $\lambda,\lambda'\in I$ be such that
		\begin{equation}\label{gj2}
			\lambda x +(1-\lambda) w=\lambda' x'+(1-\lambda')w'.
		\end{equation}
		If we have either $w,w'\in Y_1$ or $w,w'\in Y_2$, then there is nothing to show since $X$ is geometrically joinable to both $Y_1$ and $Y_2$. Without lost of generality suppose $w\in Y_1$ and $w'\in Y_2$, and so $(\ref{gj2})$ gives us
		\begin{equation*}
			X\overline{\ast} Y_1\ni \lambda x +(1-\lambda) w=\lambda' x'+(1-\lambda')w'\in X\overline{\ast} Y_2.
		\end{equation*}
		Then by $(\ref{int2})$, there are $x''\in X$, $w''\in Y_1\cap Y_2$ and $\lambda''\in I$ such that 
		\begin{align}
			\lambda x +(1-\lambda) w = & \lambda'' x''+(1-\lambda'')w''\label{gjw1},\\
			\lambda' x'+(1-\lambda')w' = & \lambda'' x''+(1-\lambda'')w''\label{gjw2}.
		\end{align}
		
		If $\lambda''=0$ (respectively $\lambda''=1$) then
		\begin{itemize}
			\item $\lambda=0$ (respectively $\lambda=1$) and $w=w''$ (respectively $x=x''$)  by $(\ref{gjw1})$, since $X$ and $Y_1$ are geometrically joinable.
			\item $\lambda'=0$ (respectively $\lambda'=1$) and $w'=w''$ (respectively $x'=x''$) by $(\ref{gjw2})$, since $X$ and $Y_2$ are geometrically joinable.
		\end{itemize}
		Thus $\lambda=0=\lambda'$ and $w=w'$ (respectively $\lambda=1=\lambda'$ and $x=x'$).
		
		If $0\neq \lambda''\neq 1$ then
		\begin{itemize}
			\item $\lambda=\lambda''$ and, $x=x''$ and $w=w''$ by $(\ref{gjw1})$, since $X$ and $Y_1$ are geometrically joinable.
			\item $\lambda'=\lambda''$ and, $x'=x''$ and $w'=w''$ by $(\ref{gjw2})$, since $X$ and $Y_2$ are geometrically joinable.
		\end{itemize}
		Thus $\lambda=\lambda'$ and, $x=x'$ and $w=w'$. Therefore $X$ is geometrically joinable to $Y_1\cup Y_2$.	
	\end{proof}

	\begin{notation} Now let $n=(k+1)m$ and for $i\in [m]$, consider the following notations:
		
		\begin{itemize}
			\item $v_i^\ell=e_{(k+1)(i-1)+\ell}$, for any $\ell=1,\cdots , k+1$, 
			\item $a_i=\dfrac{1}{k+1}\sum_{\ell=1}^{k+1}v_i=\text{bar}\{v_i^\ell\}_{\ell=1}^{k+1}$, that is, $a_i$ is the barycenter of $\{v_i^\ell\}_{\ell=1}^{k+1}$,
			\item $\Delta_i=\text{cx}\left\lbrace v_i^\ell\right\rbrace_{\ell=1}^{k+1}$,
			\item $S_i=\partial \Delta_i$.
		\end{itemize}
	\end{notation}
	
	\begin{lemma}\label{gji} For any subset $\sigma\subseteq [m]$, the collections $\{\Delta_i\}_{i\in \sigma}$, $\{S_i\}_{i\in \sigma}$ and $\{a_i\}_{i\in \sigma}$ are respectively families of geometrically joinable compact subspaces of $\R^n$.
	\end{lemma}
	
	\begin{proof} Consider the following identity of convex combinations
		\begin{equation}\label{del1}
			\sum_{i=1}^{|\sigma|}\lambda_i x_i=\sum_{i=1}^{|\sigma|}\lambda_i' x_i',
		\end{equation}
		for some $x_i,x_i'\in \Delta_i$ and $\lambda_i,\lambda_i'\in I$ with $\sum_{i=1}^{k}\lambda_i=1=\sum_{i=1}^{k}\lambda_i'$. The equation $(\ref{del1})$ is equivalent to
		\begin{equation*}
		\sum_{i=1}^{|\sigma|}\sum_{\ell=1}^{k+1}\lambda_i t_i^\ell v_i^\ell=\sum_{i=1}^{|\sigma|}\sum_{\ell=1}^{k+1}\lambda_i' s_i^\ell v_i^\ell,
		\end{equation*}
		where $x_i=\sum_{\ell=1}^{k+1} t_i^\ell v_i^\ell$ and $x_i'=\sum_{\ell=1}^{k+1} s_i^\ell v_i^\ell$ are convex combinations. Since $\{v_i^\ell: i\in \sigma, \ell=1,\cdots,k+1\}$ is affinely independent, then $\lambda_i t_i^\ell=\lambda_i' s_i^\ell$, for all $i\in \sigma, \ell=1,\cdots,k+1$. Without loss of generality if $\lambda_i\neq 0$ (the proof is similar if we assume $\lambda_i'\neq 0$), then $t_i^\ell=\dfrac{\lambda_i'}{\lambda_i} s_i^\ell$. Hence we have
		\begin{align*}
		\sum_{\ell=1}^{k+1}t_i^\ell=1 & \Rightarrow  \sum_{\ell=1}^{k+1}\dfrac{\lambda_i'}{\lambda_i} s_i^\ell=1 \\
											 & \Rightarrow 	\dfrac{\lambda_i'}{\lambda_i}\sum_{\ell=1}^{k+1} s_i^\ell=1 \\
											 & \Rightarrow 	\dfrac{\lambda_i'}{\lambda_i}=1,\ \text{since}\ \sum_{\ell=1}^{k+1} s_i^\ell=1 \\
											& \Rightarrow \lambda_i'= \lambda_i \\
											& \Rightarrow t_i^\ell= s_i^\ell\text{, for all}\ i\in \sigma, \ell=1,\cdots,k+1 \\
											& \Rightarrow x_i=\sum_{\ell=1}^{k+1} t_i^\ell v_i^\ell=\sum_{\ell=1}^{k+1} s_i^\ell v_i^\ell=x_i' \\
											& \Rightarrow x_i=x_i'.
\end{align*}
Therefore the collection $\{\Delta_i\}_{i\in \sigma}$ is geometrically joinable. As the collection of boundaries of disjoint $k$-simplices $\Delta_i$ in $\R^n$ respectively,  $\{S_i\}_{i\in \sigma}$ is a family of geometrically joinable compact subspaces of $\R^n$. Likewise, the collection of barycenters $\{a_i\}_{i\in \sigma}$ of the $k$-simplices $\Delta_i$ is geometrically joinable.
	\end{proof}
	
	\begin{notation} For any $\sigma\subseteq [m]$ and by \textbf{Lemma \ref{gji}}, consider the setting
		\begin{itemize}
			\item $J(\sigma)=\cup_{i\in \sigma}\{(k+1)(i-1)+\ell\}_{\ell=1}^{k+1}\subseteq [n]$, 
			\item $\Delta_\sigma=\Delta(J(\sigma))= \bbast_{i\in \sigma}\Delta_i$,
			\item $S_\sigma=\bbast_{i\in \sigma}S_i$, 
			\item $S_\sigma^*=\bbast_{j\notin \sigma}S_j$,
			\item $a_\sigma=\text{cx}\{a_i:i\in \sigma\}$.
		\end{itemize}
	An example of this setup is illustrated in \textsc{Figure} \ref{n2m}.
	\end{notation}
	
	\begin{figure}[h]
		\begin{center}
			\includegraphics[scale=.8]{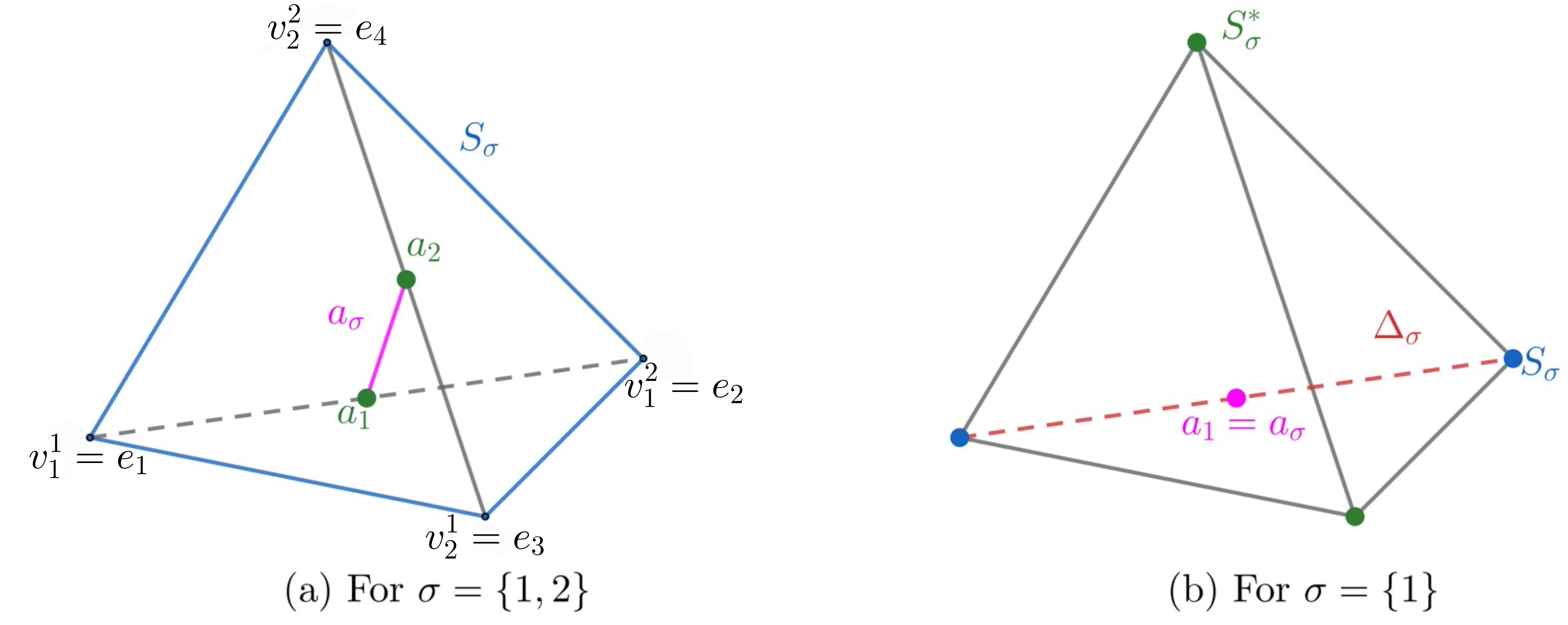}
		\end{center}
		\caption{Examples of $a_i$, $S_i$, $a_\sigma$, $S_\sigma$ and $S^*_\sigma$, for $m=2$ and $k=1$, that is $n=4$.}
		\label{n2m}
	\end{figure}
	
	\begin{lemma}\label{gjs}
		For $\sigma\subseteq [m]$, the compact spaces $a_\sigma$ and $S_\sigma$ are geometrically joinable, and we have
		\[\Delta_\sigma=a_\sigma\overline{\ast} S_\sigma.\]
	\end{lemma}
	\begin{proof}Let us prove $a_\sigma$ and $S_\sigma$ are geometrically joinable for $k=1$ and for $\sigma=\{1,2\}$; the general case can be proved similarly. In this setup, $S_\sigma$ can be split as follows
			\begin{equation*}
				S_\sigma=\underbrace{[e_1,e_3]}_{F_1}\cup \underbrace{[e_1,e_4]}_{F_2}\cup\underbrace{[e_2,e_3]}_{F_3}\cup\underbrace{[e_2,e_4]}_{F_4}\ \text{and}
			\end{equation*}   
			\[a_\sigma=\left[ \dfrac{1}{2}e_1+\dfrac{1}{2}e_2, \dfrac{1}{2}e_3+\dfrac{1}{2}e_4\right]. \]
The complexes $a_\sigma$ and $F_i$, for each $i=1,2,3,4$, are both $1$-simplices and all their four vertices are not coplanar. So $a_\sigma$ and $F_i$, for each $i=1,2,3,4$, are geometrically joinable and their join $a_\sigma\overline{\ast} F_i$ is a $3$-simplex. Also we have
			\begin{align*}
				a_\sigma \overline{\ast} (F_1\cap F_2) = & a_\sigma \overline{\ast} \{e_1\}\\
				= & (a_\sigma \overline{\ast} F_1) \cap (a_\sigma \overline{\ast} F_2).
			\end{align*}
			Hence by \textbf{Lemma \ref{gjlem}}, $a_\sigma$ is geometrically joinable to $F_1\cup F_2$. Similarly, we have
			\begin{align*}
				a_\sigma \overline{\ast} ((F_1\cup F_2)\cap F_3) = & a_\sigma \overline{\ast} \{e_3\}\\
				= & (a_\sigma \overline{\ast} (F_1\cup F_2)) \cap (a_\sigma \overline{\ast} F_3).
			\end{align*}
			Hence $a_\sigma$ is geometrically joinable to $F_1\cup F_2\cup F_3$. Similarly, we have
			\begin{align*}
				a_\sigma \overline{\ast} ((F_1\cup F_2\cup F_3)\cap F_4) = & a_\sigma \overline{\ast} \{e_2,e_4\}\\
				= & (a_\sigma \overline{\ast} (F_1\cup F_2\cup F_3)) \cap (a_\sigma \overline{\ast} F_4).
			\end{align*}
			Hence $a_\sigma$ is geometrically joinable to $S_\sigma=F_1\cup F_2\cup F_3\cup F_4$.
			
			 For any $i\in [m]$, $S_i=\partial\Delta_i$ and $a_i=\text{bar}\{v_i^\ell\}_{\ell=1}^{k+1}$. Then $a_i$ and $S_i$ are geometrically joinable, and we have $\Delta_i=a_i\overline{\ast} S_i$. We deduce
			\begin{align*}
				\Delta_\sigma & =\bbast_{i\in \sigma} \Delta_i \\
				& = \bbast_{i\in \sigma} (a_i\overline{\ast} S_i) \\
				& =( \bbast_{i\in \sigma} a_i)\overline{\ast} (\bbast_{i\in \sigma} S_i) \\
				& = a_\sigma \overline{\ast} S_\sigma. \qedhere
			\end{align*}

	\end{proof}
	
	\section{Proof of Theorem \ref{main}}\label{four}
	
	Now we have all the tools to write down the proof of  \textbf{Theorem \ref{main}} for $k\geq 1$. The case of $k=0$ will be treated in the next section. In the following, we consider the notations introduced in \textsc{Sections} 2 and 3.
	\begin{proof} Setting $W_\sigma=\Delta_\sigma\overline{\ast} S_\sigma^*$ for each $\sigma\in K$, we have
		\begin{equation}\label{w1}
			\bigcup_{\sigma\in K} W_\sigma=\bigcup_{\sigma\in K}\left( \left( \bbast_{i\in \sigma}\Delta_i\right) \overline{\ast}\left( \bbast_{j\notin \sigma}S_j\right) \right).
		\end{equation}
		Consider the collection of simplices $A(K) = \{a_{\sigma}\}_{\sigma \in K}$, which is a geometric realization of the simplicial complex $K$ and 
		\begin{equation}\label{k}
			|A(K)|=\bigcup_{\sigma\in K}a_\sigma
		\end{equation}
		to be the underlying subspace of $\R^n=\R^{(k+1)m}$. The complexes $a_{[m]}$ and $S_{[m]}$ are geometrically joinable by \textbf{Lemma \ref{gjs}}, and also $A(K)$ is a subcomplex of $a_{[m]}$. Therefore $|A(K)|$ and $S_{[m]}$ are geometrically joinable by \textbf{Remark \ref{gjss}}, and we have $S_{[m]}\overline{\ast} |A(K)|\cong S^{km-1} \ast |K| \cong \Sigma^{km}|K|$, since $S_{[m]}\cong S^{km-1}$. Then 
		\begin{align*}
			W_\sigma & = \Delta_\sigma\overline{\ast} S_\sigma^*\\
			& = a_\sigma \overline{\ast} (S_\sigma\overline{\ast} S_\sigma^*)\ \text{by \textbf{Lemma \ref{gjs}}}\\
			& = a_\sigma\overline{\ast} S_{[m]}.
		\end{align*}
	This is illustrated by \textsc{Figure} \ref{Ws}, where $W_\sigma$ is the union of the two blue triangular surfaces. In this example, we have $S_{[m]}\cong \partial W_\sigma$.
	\begin{figure}[h]
		\begin{center}
			\includegraphics[scale=.8]{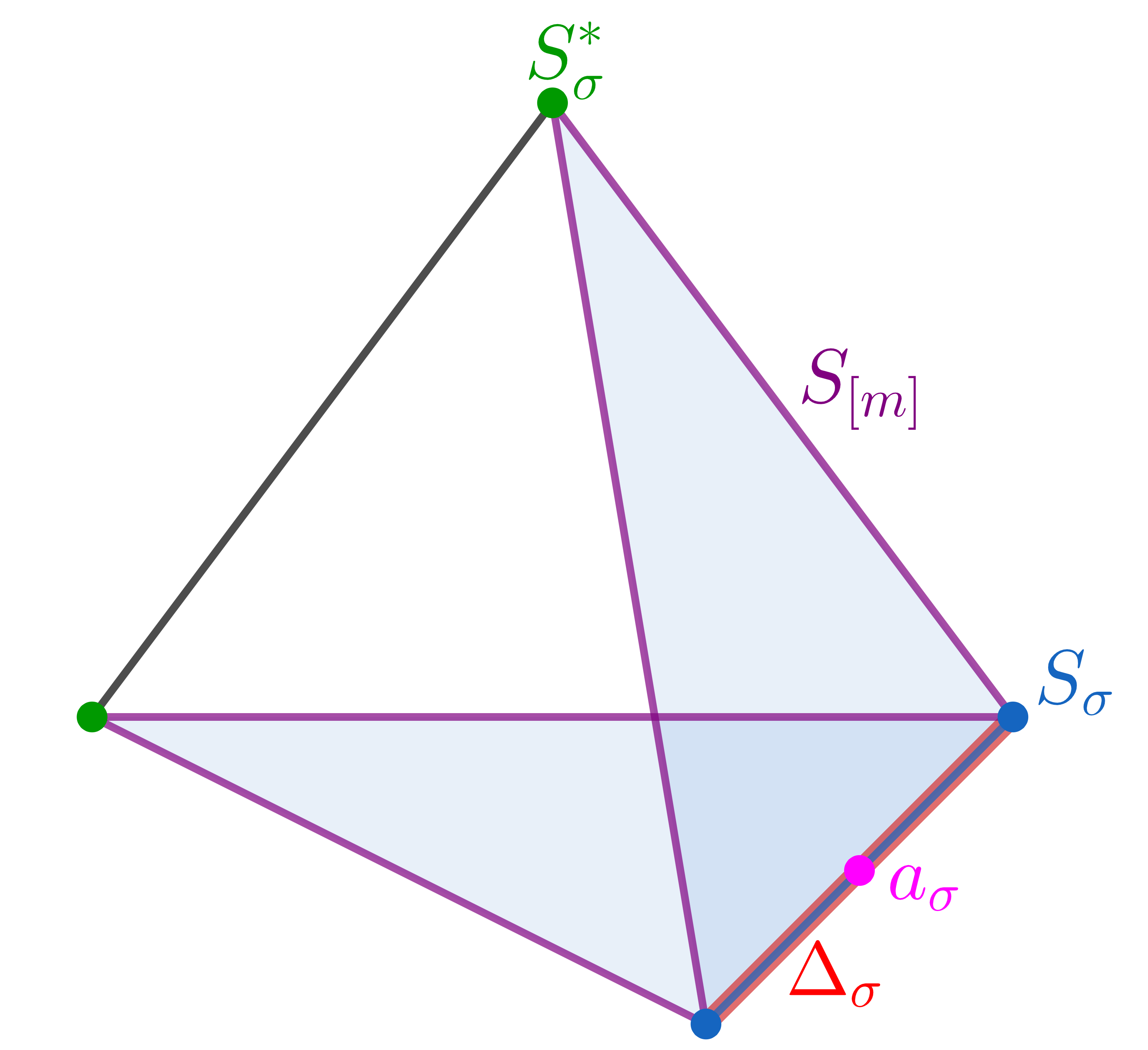}
		\end{center}
		\caption{Examples of $W_\sigma$, for $m=2$, $\sigma=\{1\}$ and $k=1$, that is $n=4$.}
		\label{Ws}
	\end{figure}
	
		Hence from equation (\ref{k}) we also have
		\begin{equation}\label{w2}
			\bigcup_{\sigma\in K} W_\sigma\cong \Sigma^{km}|K|.
		\end{equation}
		Therefore we have
		\begin{align}
			\Psi(C(\bigcup_{\sigma\in K} W_\sigma)) & \cong \bigcup_{\sigma\in K} \Psi(CW_\sigma) \nonumber\\
			& = \bigcup_{\sigma\in K} \Psi\left( C\left( \left( \bbast_{i\in \sigma}\Delta_i\right) \overline{\ast}\left( \bbast_{j\notin \sigma}S_j\right) \right)\right) \text{by (\ref{w1})}\nonumber \\
			& \cong \bigcup_{\sigma\in K} \Psi\left( \prod_{i\in \sigma}(C\Delta_i)\times \prod_{j\notin \sigma}(CS_j)\right) \text{by \textbf{Lemma \ref{bbcg}}}\nonumber\\
			& \cong \bigcup_{\sigma\in K} \Psi\left( \prod_{i\in \sigma}C_i^{k+1}\times \prod_{j\notin \sigma}\partial_- C_j^{k+1}\right)\ \text{by \textbf{Lemma \ref{psi}},} \label{pw1}\\
			& \qquad \text{where $C_i^{k+1}$ and $C_j^{k+1}$ are copies of the $(k+1)$-cube.} \nonumber
		\end{align}
		 Then we get
		\begin{align}
			\overline{\Psi}(\Sigma \bigcup_{\sigma\in K}W_\sigma) & \cong \bigcup_{\sigma\in K} \overline{\Psi}(\Sigma W_\sigma) \nonumber\\
			&  \cong \bigcup_{\sigma\in K}\ \omega\left( \prod_{i\in \sigma}C_i^{k+1}\times \prod_{j\notin \sigma}\partial_- C_j^{k+1}\right)  \text{by}\ (\ref{pw1})\ \text{and}\ \textbf{Lemma \ref{psi}}\nonumber\\
			&  \cong \bigcup_{\sigma\in K}\ \left(\bigwedge_{i\in \sigma}\widetilde{D}_i^{k+1}\wedge \bigwedge_{j\notin \sigma}\partial \widetilde{D}_i^{k+1}\right) \text{by \textbf{Remark \ref{om}}, where $\widetilde{D}_i^{k+1}$ are copies of}\nonumber\\
			& \qquad \text{the nonstandard $(k+1)$-disk $\widetilde{D}^{k+1}$ defined in \textbf{Lemma \ref{nsdisk}}.}\nonumber\\
			\overline{\Psi}(\Sigma \bigcup_{\sigma\in K}W_\sigma) & \cong \bigcup_{\sigma\in K}\ \left(\bigwedge_{i\in \sigma}D_i^{k+1}\wedge \bigwedge_{j\notin \sigma}S_j^k\right)\ \text{by \textbf{Lemma \ref{nsdisk}}, where $D_i^{k+1}$ and $S_j^{k}$ are copies of}\nonumber\\
			& \qquad \text{the $(k+1)$-disk $D^{k+1}$ and the $k$-sphere $S^k$ respectively}\nonumber\\
			 & \cong \widehat{Z}(K; (D^{k+1},S^k)). \label{pw2}
		\end{align}
		On the other hand, we obtain
		\begin{align}
			\overline{\Psi}(\Sigma \bigcup_{\sigma\in K}W_\sigma) & \cong \overline{\Psi}(\Sigma \Sigma^{km}|K|)\ \text{by (\ref{w2})} \nonumber\\
			& \cong \overline{\Psi}\left( \Sigma^{km+1}|K|\right) \nonumber\\
			& \cong \Sigma^{km+1}|K|\ \text{by \textbf{Lemma \ref{psi}}.} \label{pw3}
		\end{align}
		Therefore by $(\ref{pw2})$ and $(\ref{pw3})$, we have
		\begin{equation*}
			\widehat{Z}(K; (D^{k+1},S^k))\cong \Sigma^{km+1}|K|.
		\end{equation*}
	Hence we have the result.
	\end{proof}

	\section{The case $k=0$}\label{five}
	In the previous section, we have proved \textbf{Theorem \ref{main}} for $k\geq 1$. Here we prove the remaining case, namely $k=0$, given by the following.
	
	\begin{theorem}\label{zero} There is a homeomorphism
	\begin{equation*}
			\widehat{Z}(K; (D^{1},S^0))\cong \Sigma |K|.
	\end{equation*}
	\end{theorem}

	For the purpose of the proof, we will consider the categorical definition of the polyhedral smash product given by $(\ref{psp})$. As in the third bullet of \textbf{Notation \ref{ch}}, consider the functor 
	\begin{align*}
		\Delta: \textsc{Cat}(K) & \to \text{Top} \\
		\sigma  & \mapsto \Delta(\sigma)\cong \Delta^{|\sigma|-1}.
	\end{align*}
The geometric realization of $K$ is given by
\begin{equation}\label{grK}
	|K| = \colim_{\sigma \in K} \Delta(\sigma).
\end{equation}

	\begin{proof}
			Consider the composite functor $\Sigma \Delta$ given by
			\begin{align*}
				\Sigma \Delta: \textsc{Cat}(K) & \to \text{Top} \\
				\sigma  & \mapsto \Sigma \Delta(\sigma)\cong \Sigma \Delta^{|\sigma|-1}.
			\end{align*}
		Let $\tau \subseteq \sigma$ be a face inclusion in $K$, with $|\tau| = p \leq \ell = |\sigma|$. We will look at the case $\tau=\{1,\cdots,p\} \subseteq \{1,\cdots,\ell\}=\sigma$ for simplicity; the same argument works for the general case. Consider the two following maps

			      	  \begin{align*}
			      	  	\phi_1: C \Delta^{p-1} & \to C \Delta^{\ell-1} \\
			      	           [c,(t_1,\cdots,t_p),\lambda] & \mapsto [c,(t_1,\cdots,t_p,\underbrace{0,\cdots,0}_{\ell-p\ \text{times}}),\lambda]
			      	  \end{align*}  and 	     
		      \begin{align*}
		      		\phi_2: C^p & \to C^\ell \\
		      	(t_1,\cdots,t_p) & \mapsto (t_1,\cdots,t_p,\underbrace{0,\cdots,0}_{\ell-p\ \text{times}}).
		      \end{align*}
			    
	      The transformation $\overline{\Psi}:  \Sigma \Delta \Rightarrow \widehat{D}$, where the functor $\widehat{D}$ is defined by $(\ref{dcat})$ for the pair $(D^1,S^0)$, is a natural isomorphism if the left hand side  diagram commutes since it induces the commutative diagram on the right hand side below:
	      \begin{center}
	      	\begin{tabular}{ccc}
	      	\includegraphics[scale=.8]{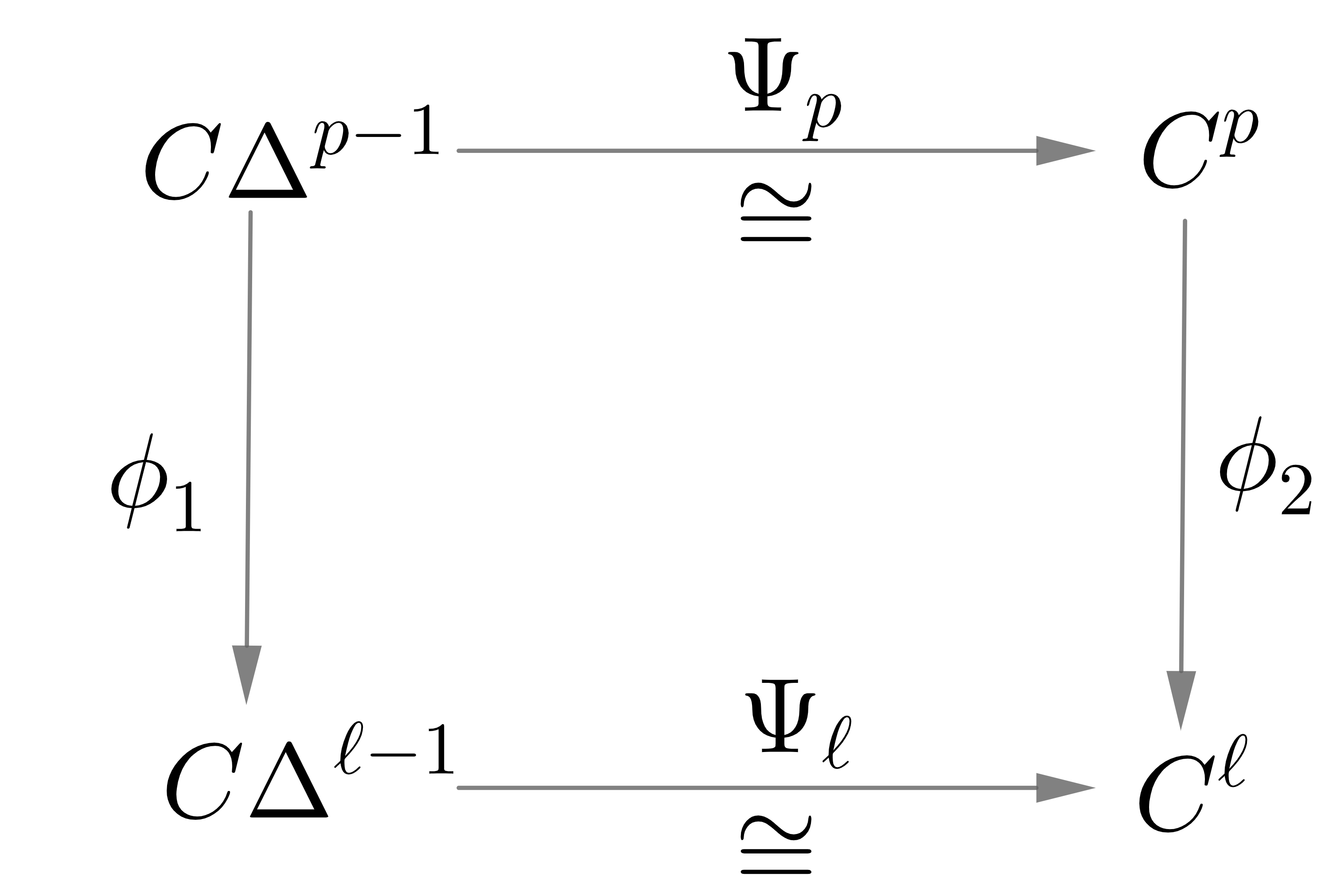} & \includegraphics[scale=.8]{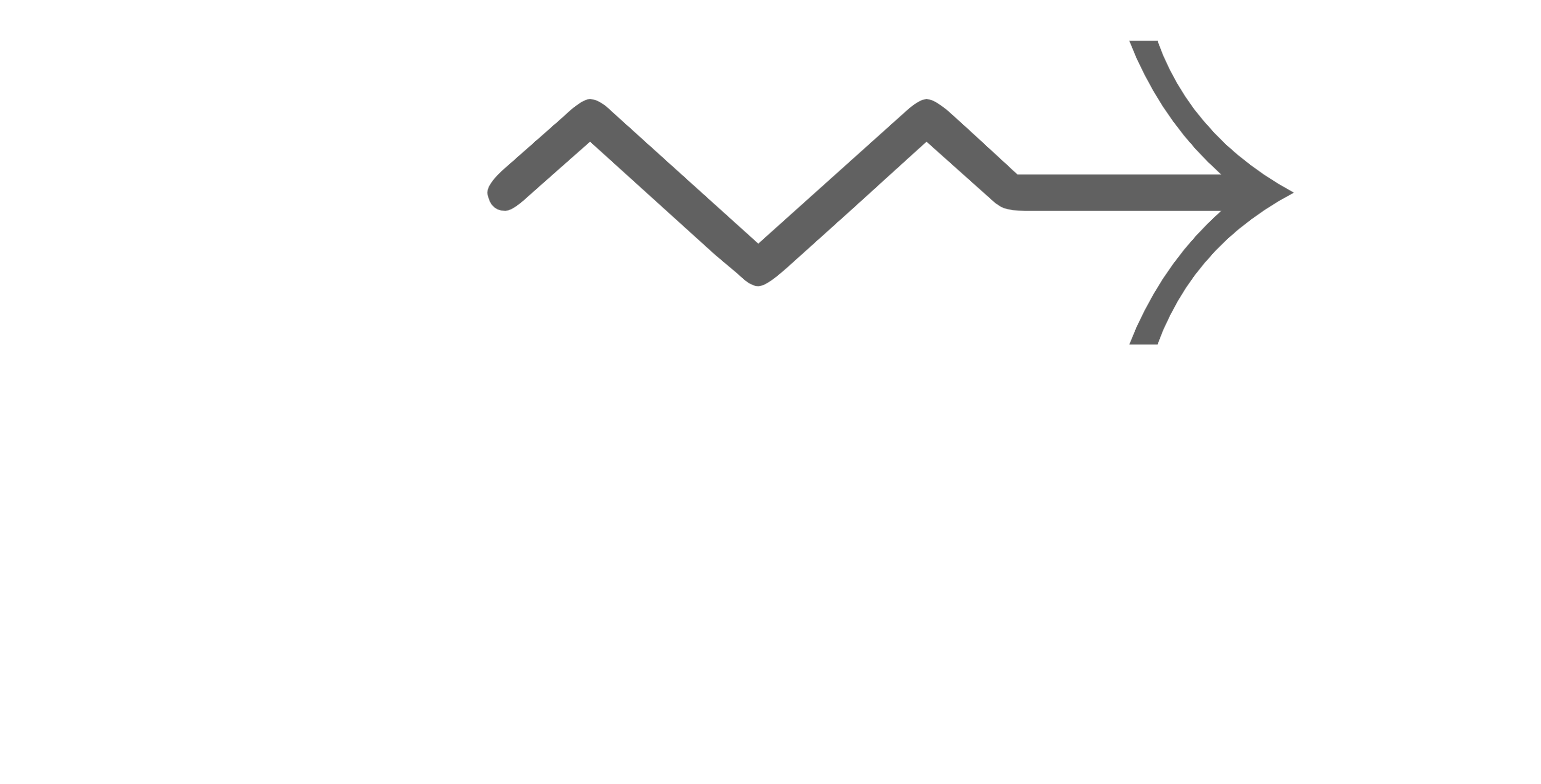} & \includegraphics[scale=.8]{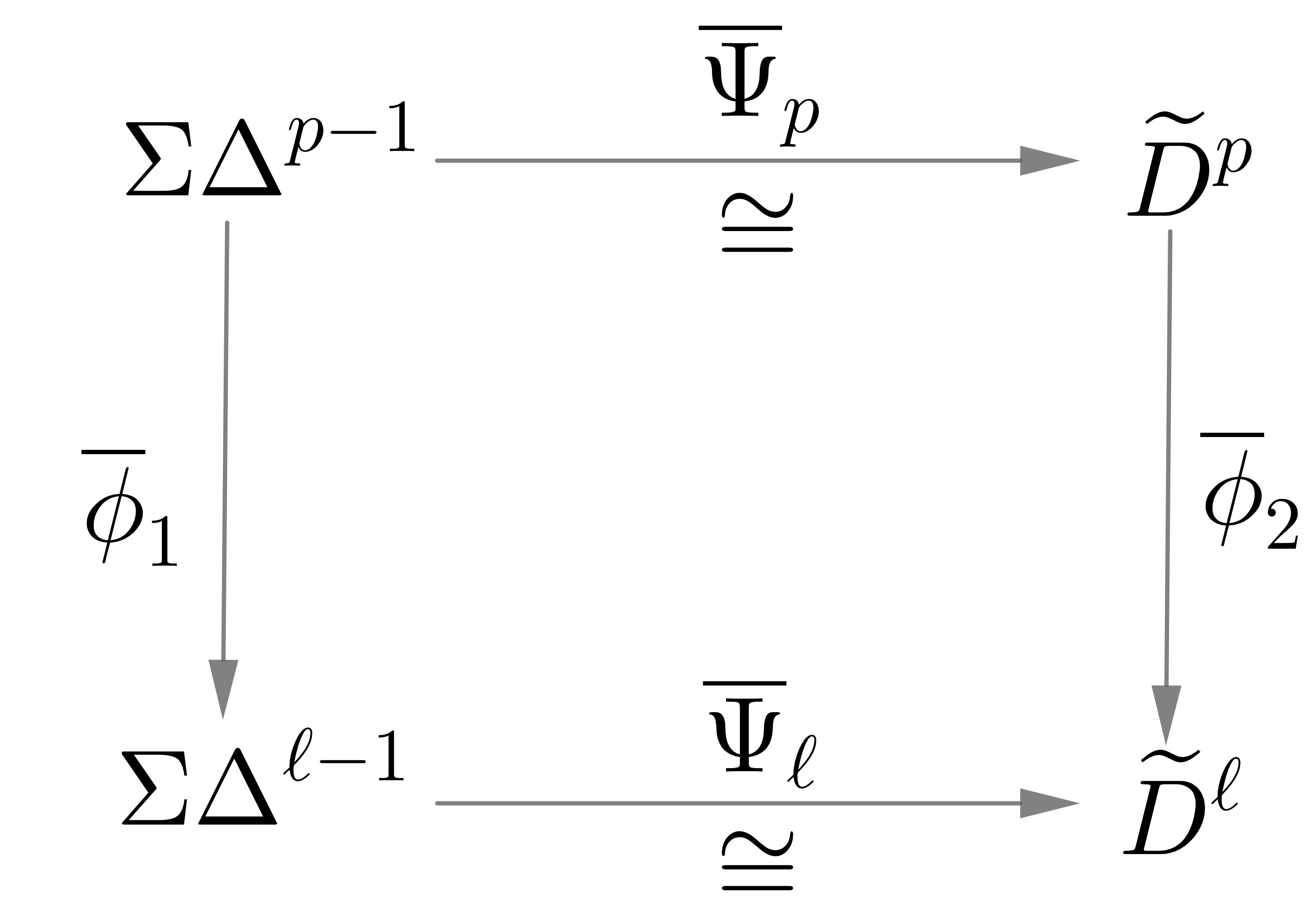}
	      \end{tabular}
	      \end{center}
	     where $\overline{\phi}_1([s_i,x,\lambda])=[\phi_1([c,x,\lambda])]$, for all $[s_i,x,\lambda]\in \Sigma \Delta^{p-1}$, $\widetilde{D}^p=\underbrace{C^1\wedge \cdots\wedge C^1}_{p\ \text{times}}$ and so\\
	      $\overline{\phi}_2([y])=[\phi_2(y)]$, for all $y\in C^{p}$. For $[c,(t_1,\cdots,t_p),\lambda]\in C \Delta^{p-1}$ we have 
	     
	      \begin{align*}
	      	\Psi_\ell \phi_1 ([c,(t_1,\cdots,t_p),\lambda]) & = 	\Psi_\ell([c,(t_1,\cdots,t_p,\underbrace{0,\cdots,0}_{\ell-p\ \text{times}}),\lambda]) \\
	      																	& = \begin{cases}
	      																		2\lambda (t_1,\cdots,t_p,\underbrace{0,\cdots,0}_{\ell-p\ \text{times}}),\ \text{if}\ 0\leq \lambda\leq \dfrac{1}{2}\\
	      																		\left( (2-2\lambda)+(2\lambda-1)\dfrac{2}{\max_{1\leq i\leq p} \{t_i\}}\right) (t_1,\cdots,t_p,\underbrace{0,\cdots,0}_{\ell-p\ \text{times}}),\ \text{if}\ \dfrac{1}{2}\leq \lambda\leq 1
	      																	\end{cases}\text{by (\ref{Psidef})}
      																	\end{align*}
      		\begin{align*}
      		 \phi_2 \Psi_p ([c,(t_1,\cdots,t_p),\lambda]) & = \begin{cases}
      		 	 \phi_2(2\lambda (t_1,\cdots,t_p)),\ \text{if}\ 0\leq \lambda\leq \dfrac{1}{2}\\
      		 	 \phi_2\left( \left( (2-2\lambda)+(2\lambda-1)\dfrac{2}{\max_{1\leq i\leq p} \{t_i\}}\right)  (t_1,\cdots,t_p)\right),\ \text{if}\ \dfrac{1}{2}\leq \lambda\leq 1\\
      	  	\end{cases}\text{by (\ref{Psidef})}\\
        	& = \begin{cases}
        		2\lambda (t_1,\cdots,t_p,\underbrace{0,\cdots,0}_{\ell-p\ \text{times}}),\ \text{if}\ 0\leq \lambda\leq \dfrac{1}{2}\\
        		\left( (2-2\lambda)+(2\lambda-1)\dfrac{2}{\max_{1\leq i\leq p} \{t_i\}}\right) (t_1,\cdots,t_p,\underbrace{0,\cdots,0}_{\ell-p\ \text{times}}),\ \text{if}\ \dfrac{1}{2}\leq \lambda\leq 1.
        	\end{cases}
	      \end{align*}
      Then the two diagrams commute and therefore $\overline{\Psi}$ is a natural isomorphism. Passing to the colimit, we have
      \begin{equation}\label{col}
     \colim \overline	{\Psi} : \colim_{\sigma\in K} \Sigma \Delta(\sigma) \stackrel{\cong}{\longrightarrow} \colim_{\sigma\in K} \widehat{D}(\sigma).
 \end{equation}
  But we have
  \[\colim_{\sigma\in K} \Sigma \Delta(\sigma) \cong \Sigma \colim_{\sigma\in K}  \Delta(\sigma) = \Sigma |K|\ \text{by (\ref{grK})}\]
   and also considering identity $(\ref{psp})$, the homeomorphism $(\ref{col})$ yields a homeomorphism
  \[\Sigma |K|\cong \widehat{Z}(K; (\underline{X},\underline{A})) .\qedhere\]
	  \end{proof}    

\section{Generalization of  Theorem \ref{main}}\label{six}

In this section, we generalize \textbf{Theorem \ref{main}} further, using an argument kindly provided by the referee. Instead of doubling all the vertices of $K$ simultaneously as in David Stone's original construction, we double one vertex at a time and argue inductively, starting from the case $k=0$ in \textbf{Theorem \ref{zero}}. Let us start by setting some notation and stating intermediate results. 

\begin{notation}\text{\ } 
	
	\begin{itemize}
	\item[$\bullet$] For $J=(j_1,\cdots,j_m)$ an $m$-tuple from $(\N\cup\{0\})^m$, denote the family of CW-pairs
 \[\left( \underline{D^{J+1}},\underline{S^J}\right) =\left\lbrace \left( D^{j_1+1}, S^{j_1}\right) ,\cdots,\left( D^{j_m+1},S^{j_m}\right) \right\rbrace .\]
	\item[$\bullet$] Set $J_i=(0,\cdots,1,\cdots,0)$ to be the $m$-tuple having $1$ only at the $i$-th position and $0$ elsewhere. For a simplicial complex $K$ over $[m]$ and $i\in [m]$, consider the new simplicial complex $K(J_i)$ with $m+1$ vertices labeled $\{1, \cdots, i-1, i_a, i_b, i+1, \cdots, m\}$ and defined by
\begin{align*}
	K(J_i) & :=\ \{(\sigma\backslash\{i\})\cup\{i_a, i_b\}\ |\ \sigma\in K\text{ and }i\in \sigma\}  \cup \{\sigma\cup\{i_a\}\ |\ \sigma\in K\text{ and }i\notin \sigma\}  \\
			  &\quad  \cup \{\sigma\cup\{i_b\}\ |\ \sigma\in K\text{ and }i\notin \sigma\} \cup \{\text{all their subsets}\}.
\end{align*}
	\end{itemize}
\end{notation} 
The meaning behind the introduction of $K(J_i)$ is illustrated in the following example.
\begin{example} For $m=2$, consider $K=\{\emptyset, \{1\},\{2\}\}$ and $J_1=(1,0)$. We have 
	\begin{align*}
		\widehat{Z}(K;(\underline{D^{J_1+1}},\underline{S^{J_1}})) & = \widehat{Z}(K;\{(D^{2},S^{1}), (D^{1},S^{0})\}) \\
		& = \widehat{D}(\{1\})\cup \widehat{D}(\{2\})\quad \text{as a subspace of}\ D^2 \wedge D^1 \\
		& = D^2\wedge S^0 \cup S^1\wedge D^1 \\
		& \cong (D^1\wedge D^1)\wedge S^0 \cup (D^1\wedge S^0 \cup S^0 \wedge D^1)\wedge D^1 \\ 
		& = \left( D^1\wedge D^1\wedge S^0\right)  \cup \left( D^1\wedge S^0\wedge D^1\right)  \cup \left( S^0 \wedge D^1\wedge D^1\right) \\ 
		& = \widehat{D}(\{1_a,1_b\})\cup\widehat{D}(\{2\}\cup\{1_a\})\cup\widehat{D}(\{2\}\cup\{1_b\}) \\
		& = \widehat{Z}(K(J_1);\{(D^{1},S^{0}),(D^{1},S^{0}), (D^{1},S^{0})\}) \\
		& = \widehat{Z}(K(J_1);(D^{1},S^{0})),\ \text{see}\ \textsc{Figure}\ \ref{e63}.
	\end{align*} 
\begin{figure}[h]
	\begin{center}
		\includegraphics[scale=.8]{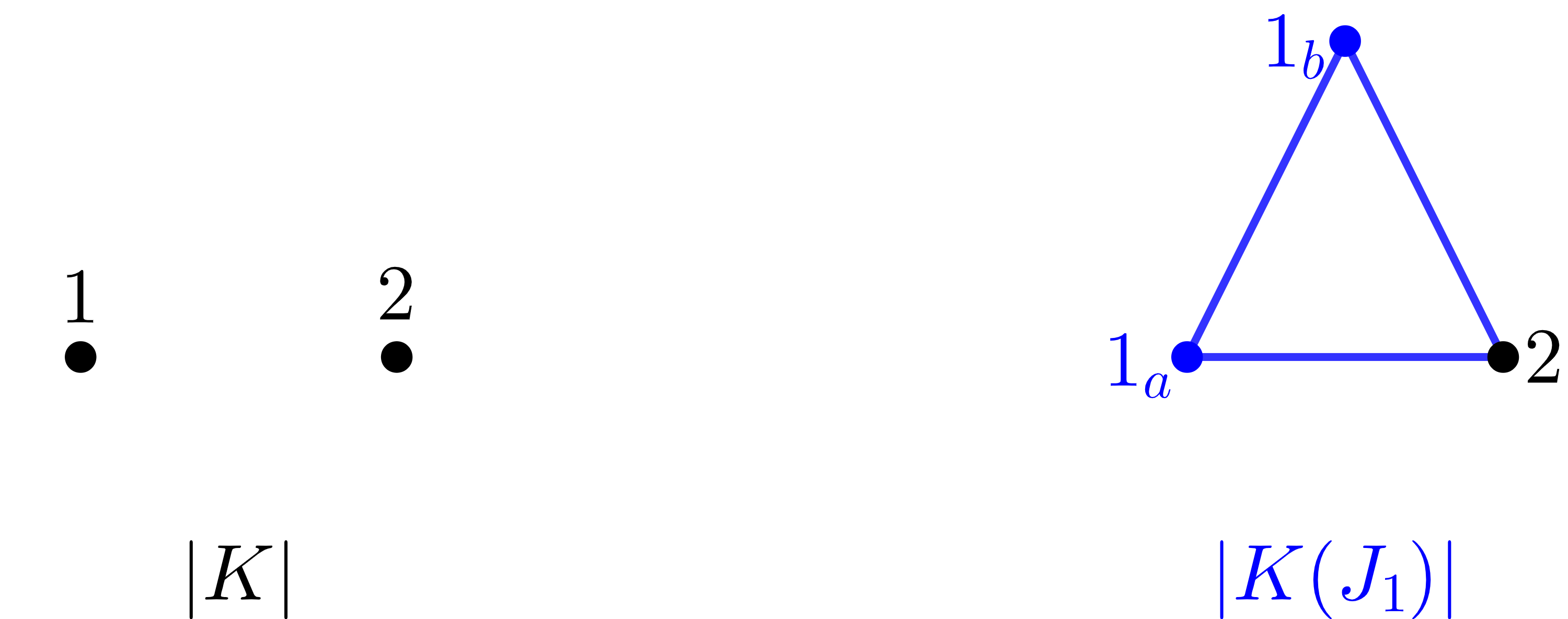}
	\end{center}
	\caption{$K(J_1) = \{\{1_a, 1_b\}, \{1_a, 2\}, \{1_b, 2\}, \text{their subsets}\}$.}
	\label{e63}
\end{figure}
\end{example}
The next lemma suggests that the polyhedral smash product $\widehat{Z}(K;(\underline{D^{J+1}},\underline{S^{J}}))$ can be computed iteratively with steps involving $K(J_i)$ for some $i\in [m]$.

\begin{lemma}\label{Ji} Let $J=(j_1,\cdots,j_m)$ to be an $m$-tuple and $i\in [m]$ such that $j_i\neq 0$. There is a homeomorphism
\begin{equation*}
	\widehat{Z}(K;(\underline{D^{J+1}},\underline{S^{J}}))\cong \widehat{Z}(K(J_i);(\underline{D^{J'+1}},\underline{S^{J'}})),
\end{equation*}
where $J'$ is the $(m+1)$-tuple $J'=(j_1,\cdots,j_i-1,0,\cdots,j_m)$.
\end{lemma}
\begin{proof} The polyhedral smash product $\widehat{Z}(K;(\underline{D^{J+1}},\underline{S^{J}}))$ is defined as follows
	\begin{align*}
		\widehat{Z}(K;(\underline{D^{J+1}},\underline{S^{J}})) & = \widehat{Z}(K;\{(D^{j_1+1},S^{j_1}),\cdots,(D^{j_i+1},S^{j_i}),\cdots, (D^{j_m+1},S^{j_m})\}) \\
		& = \bigcup_{\sigma\in K}\left( \widehat{D}(\sigma) \right) \quad \text{as a subspace of}\  \bigwedge_{\ell=1}^m D^{j_\ell+1}\cong \widetilde{D}^{j_1+\cdots+j_m+m}\\
		& = \bigcup_{\sigma\in K} \left( \bigwedge_{\ell=1}^m Y_\ell\right),\ \text{where}\ Y_\ell\ \text{depends on}\ \sigma\ \text{and is defined in}\ (\ref{dhat})  \\
		& = \bigcup_{\sigma\in K, i\in \sigma} \left( \bigwedge_{\ell=1}^m Y_\ell\right)  \cup \bigcup_{\sigma\in K, i\notin \sigma} \left( \bigwedge_{\ell=1}^m Y_\ell\right) \\
		& = \bigcup_{\sigma\in K, i\in \sigma} \left( Y_1\wedge\cdots\wedge D^{j_i+1}\wedge\cdots \wedge Y_m\right)  \cup \bigcup_{\sigma\in K, i\notin \sigma} \left( Y_1\wedge\cdots\wedge S^{j_i}\wedge\cdots\wedge Y_m\right)  \\
		& \cong \bigcup_{\sigma\in K, i\in \sigma} \left( Y_1\wedge\cdots\wedge (D^{j_i}\wedge D^1)\wedge\cdots Y_m\right)  \\
		& \quad \cup \bigcup_{\sigma\in K, i\notin \sigma} \left( Y_1\wedge\cdots\wedge (D^{j_i}\wedge S^0\cup S^{j_i-1}\wedge D^1)\wedge\cdots Y_m\right)  \\
		& = \bigcup_{\sigma\in K, i\in \sigma} \left( Y_1\wedge\cdots\wedge D^{j_i}\wedge D^1\wedge\cdots Y_m\right) \cup \bigcup_{\sigma\in K, i\notin \sigma} \left( Y_1\wedge\cdots\wedge  D^{j_i}\wedge S^0\wedge\cdots Y_m\right)  \\
		& \quad \cup \bigcup_{\sigma\in K, i\notin \sigma} \left( Y_1\wedge\cdots\wedge S^{j_i-1}\wedge D^1\wedge\cdots Y_m\right) \\
		& = \bigcup_{\sigma\in K, i\in \sigma} \widehat{D}\left( (\sigma\backslash\{i\})\cup\{i_a, i_b\}\right) \cup \bigcup_{\sigma\in K, i\notin \sigma} \widehat{D}\left( \sigma\cup\{i_a\}\right) \cup \bigcup_{\sigma\in K, i\notin \sigma} \widehat{D}\left( \sigma\cup\{i_b\}\right)\\
		& = \widehat{Z}\left( K(J_i);\{(D^{j_1+1},S^{j_1}),\cdots,(D^{(j_i-1)+1},S^{j_i-1}),(D^1,S^0),\cdots, (D^{j_m+1},S^{j_m})\}\right)  \\
		& = \widehat{Z}(K(J_i);(\underline{D^{J'+1}},\underline{S^{J'}})),
	\end{align*}
where $J'$ is the $(m+1)$-tuple $J'=(j_1,\cdots,j_i-1,0,\cdots,j_m)$.
\end{proof}
\begin{example}\label{ex64}\text{\ }
	
\begin{enumerate}
	\item For $m=3$, consider $K=\{\{1,2\},\{3\},\text{their subsets}\}$ and $J=(1,1,0)$. We have 
	\begin{align*}
				\widehat{Z}(K;(\underline{D^{J+1}},\underline{S^{J}})) & = \widehat{Z}(K;\{(D^{2},S^{1}),(D^{2},S^{1}), (D^{1},S^{0})\}) \\
				& = \widehat{D}(\{1,2\})\cup \widehat{D}(\{3\}) \quad \text{as a subspace of}\  D^2\wedge D^2\wedge D^1 \\
				& = D^2\wedge D^2 \wedge S^0 \cup S^1\wedge S^1 \wedge D^1 \\
				& \cong D^2\wedge (D^1\wedge D^1)\wedge S^0 \cup S^1\wedge (D^1\wedge S^0 \cup S^0 \wedge D^1)\wedge D^1 \\ 
				& =  \left( D^2\wedge D^1\wedge D^1\wedge S^0\right)  \cup \left( S^1\wedge D^1\wedge S^0\wedge D^1\right) \cup \left( S^1\wedge S^0 \wedge D^1\wedge D^1\right) \\
				& = \widehat{D}((\{1,2_a,2_b\})\cup\widehat{D}(\{2_a,3\})\cup\widehat{D}(\{2_b,3\}) \\
				& = \widehat{Z}(K(J_2);\{(D^{2},S^{1}),(D^{1},S^{0}),(D^{1},S^{0}), (D^{1},S^{0})\}) \\
				& = \widehat{Z}(K(J_2);(\underline{D^{J_1+1}},\underline{S^{J_1}})),\ \text{with}\ J_1=(1,0,0,0)\\
				& = \widehat{Z}(K(J_2)(J_1);(D^{1},S^{0}))\\
				& = \widehat{Z}(K(J);(D^{1},S^{0})),\ \text{where}\ K(J)=K(J_2)(J_1).
			\end{align*} 
	See \textsc{Figure} \ref{e651}.
	\begin{figure}[h]
		\begin{center}
			\includegraphics[scale=.8]{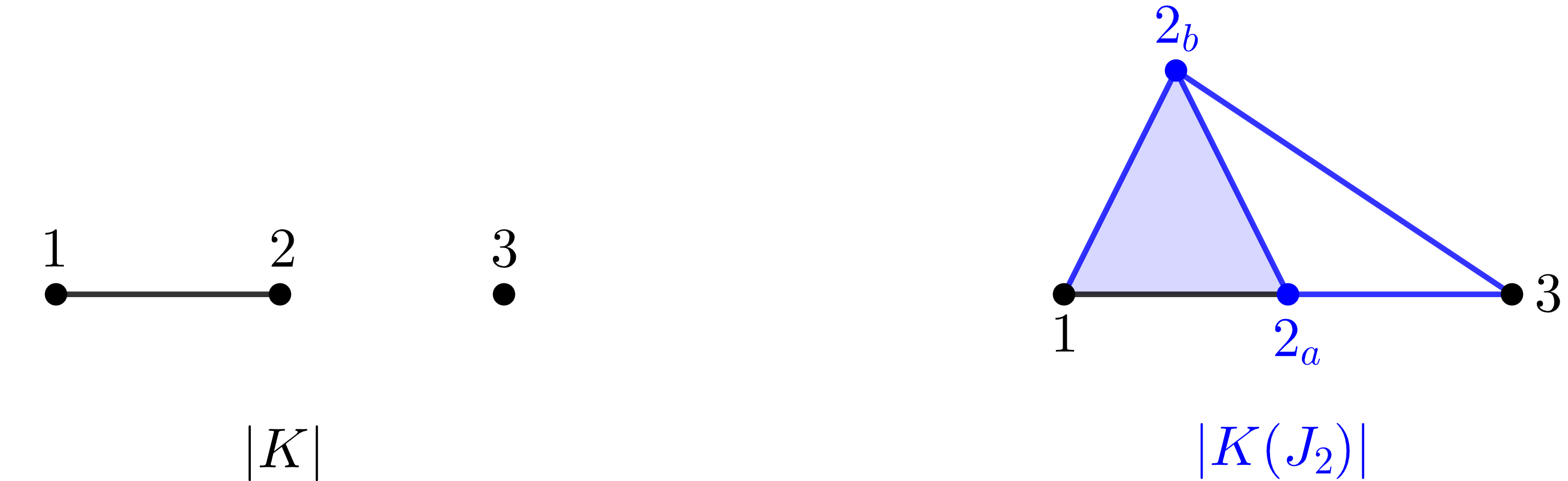}
		\end{center}
		\caption{$K(J_2) = \{\{1, 2_a, 2_b\}, \{2_a,3\}, \{2_b,3\}, \text{their subsets}\}$.}
		\label{e651}
	\end{figure}
	\item For $m=3$, consider $K=\{\{1,2\},\{1,3\},\{2,3\},\text{their subsets}\}$ and $J=(2,1,1)$. We have
	 \begin{align*}
					\widehat{Z}(K;(\underline{D^{J+1}},\underline{S^{J}})) & = \widehat{Z}(K;\{(D^{3},S^{2}),(D^{2},S^{1}), (D^{2},S^{1})\}) \\
					& \cong \widehat{Z}(K(J_1);(\underline{D^{J'+1}},\underline{S^{J'}})),\ \text{where}\ J'=(1,0,1,1).
				\end{align*} 
	With $K(J_1) = \{\{1_a, 1_b, 2\}, \{1_a,1_b,3\}, \{1_a,2,3\},\{1_b,2,3\},  \text{their subsets}\}$, which is a simplicial complex with $m+1=4$ vertices.
\end{enumerate}
\end{example}

The second intermediate result in given by the following lemma, which states that the geometric realization of the simplicial complex $K(J_i)$ can be obtained just by considering a single suspension of the geometric realization of the simplicial complex $K$.
\begin{lemma}\label{Jisus} For any $i\in [m]$, we have
\begin{equation*}
		|K(J_i)|\cong \Sigma |K|.
\end{equation*}
\end{lemma}
\begin{proof} Set $S^0=\{s_1,s_2\}$ to be the $0$-sphere. We have
	\begin{align*}
		\Sigma |K| & = S^0 \ast |K| \\
						& = \{s_1,s_2\}\ast\left( \bigcup_{\sigma\in K} |\sigma|\right) \\
						& = \bigcup_{\sigma\in K} \left( \{s_1,s_2\}\ast |\sigma|\right) \\
						& = \left( \bigcup_{\sigma\in K, i\in \sigma} \left( \{s_1,s_2\}\ast |\sigma|\right) \right) \cup \left( \bigcup_{\sigma\in K, i\notin \sigma} \left( \{s_1,s_2\}\ast |\sigma|\right) \right) \\
						& \cong \left( \bigcup_{\sigma\in K, i\in \sigma} \left( |\sigma\backslash\{i\}\cup \{s_1,s_2\}|\right) \right) \cup \left( \bigcup_{\sigma\in K, i\notin \sigma} \left( |\sigma\cup\{s_1\}|\cup |\sigma\cup\{s_2\}| \right) \right) \\
						& = \left( \bigcup_{\sigma\in K, i\in \sigma} \left( |\sigma\backslash\{i\}\cup \{s_1,s_2\}|\right) \right) \cup \left( \bigcup_{\sigma\in K, i\notin \sigma} \left( |\sigma\cup\{s_1\}| \right) \right)\cup \left( \bigcup_{\sigma\in K, i\notin \sigma} \left( |\sigma\cup\{s_2\}| \right) \right) \\
						& \cong \left( \bigcup_{\sigma\in K, i\in \sigma} \left( |\sigma\backslash\{i\}\cup \{i_a,i_b\}|\right) \right) \cup \left( \bigcup_{\sigma\in K, i\notin \sigma} \left( |\sigma\cup\{i_a\}| \right) \right)\cup \left( \bigcup_{\sigma\in K, i\notin \sigma} \left( |\sigma\cup\{i_b\}| \right) \right) \\
						& =  \bigcup_{\tau\in K(J_i)} |\tau| \\
						& = |K(J_i)|. \qedhere
	\end{align*}
\end{proof}
The lemma is illustrated in \textsc{Figure} \ref{e66} for the simplicial complex $K$ from \textbf{Example \ref{ex64}}(1).
\begin{figure}[h]
		\begin{center}
			\includegraphics[scale=.8]{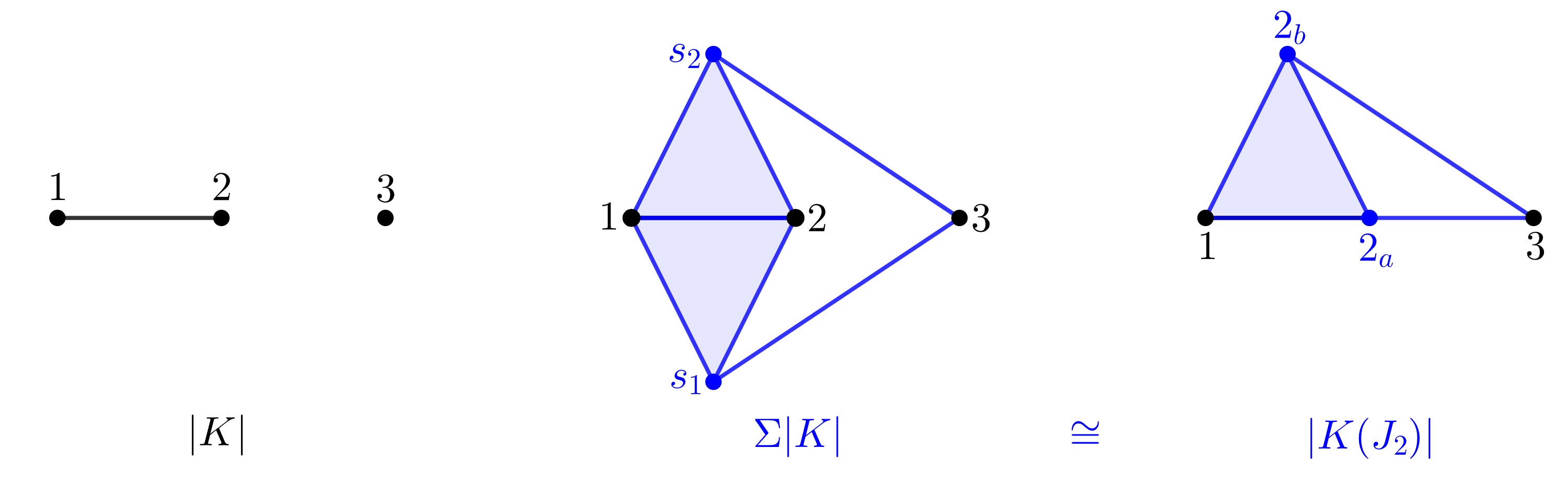}
		\end{center}
		\caption{$|K(J_2)|\cong \Sigma |K|$.}
		\label{e66}
	\end{figure}

Now we can state and prove the main result.

\begin{theorem}\label{gen} For any $m$-tuple $J=(j_1,\cdots,j_m)$ in $(\N\cup\{0\})^m$, there is a homeomorphism
	\begin{equation*}
		\widehat{Z}(K;(\underline{D^{J+1}},\underline{S^J}))\cong \Sigma^{j_1+\cdots+j_m+1} |K|.
	\end{equation*}
\end{theorem}

\begin{proof}
	Applying \textbf{Lemma \ref{Ji}} $\sum_{i=1}^{m}j_i$ times, we get
	\begin{equation}\label{gen1}
		\widehat{Z}(K;(\underline{D^{J+1}},\underline{S^{J}})) \cong \widehat{Z}(K(J);(\underline{D^{1}},\underline{S^{0}})),
	\end{equation}
where $K(J)$ is a simplicial complex obtained by applying the basic move (doubling a single vertex) $\sum_{i=1}^{m}j_i$ times. By the base case $k=0$ in \textbf{Theorem \ref{zero}}, we have 
\begin{equation}\label{gen2}
	\widehat{Z}(K(J);(\underline{D^{1}},\underline{S^{0}}))\cong \Sigma |K(J)|.
\end{equation}

Finally, by applying \textbf{Lemma \ref{Jisus}} $\sum_{i=1}^{m}j_i$ times, we have
\begin{align}\label{gen3}
	\Sigma |K(J)| & \cong \Sigma\Sigma^{j_1+\cdots+j_m} |K| \nonumber \\
						& = \Sigma^{j_1+\cdots+j_m+1} |K|.
\end{align}

Therefore by putting equations $(\ref{gen1}), (\ref{gen2})$ and $(\ref{gen3})$ together, we obtain\\
\[\widehat{Z}(K;(\underline{D^{J+1}},\underline{S^J}))\cong \Sigma^{j_1+\cdots+j_m+1} |K|. \qedhere \]
\end{proof}
	\bibliographystyle{plain}
	\bibliography{MyBib}
	
\end{document}